\documentclass[11pt]{article}  

\usepackage[margin=1in]{geometry}  
\usepackage{amsmath,amssymb,amsthm,mathtools}  
\usepackage{bm}  
\usepackage{graphicx}
\usepackage{caption}
\usepackage{microtype}  
\usepackage{enumitem}  
\usepackage[numbers]{natbib}
\usepackage{doi}
\usepackage{hyperref}
\usepackage{placeins}
\usepackage{csquotes}
\usepackage{cleveref}

\hypersetup{colorlinks=true, linkcolor=blue, citecolor=blue, urlcolor=blue, pdftitle={TSP integrality gap via 2-edge-connected multisubgraph problem under coincident IP optima}, pdfauthor={Toshiaki Yamanaka}}

\theoremstyle{plain}
\newtheorem{theorem}{Theorem}[section]
\newtheorem{lemma}[theorem]{Lemma}
\newtheorem{corollary}[theorem]{Corollary}
\newtheorem{proposition}[theorem]{Proposition}

\theoremstyle{definition}
\newtheorem{definition}[theorem]{Definition}

\newtheorem{problem}[theorem]{Problem}

\theoremstyle{remark}
\newtheorem{remark}[theorem]{Remark}

\title{TSP integrality gap via 2-edge-connected multisubgraph problem under coincident IP optima}  
\author{Toshiaki Yamanaka \thanks{Whiting School of Engineering, Johns Hopkins University, Baltimore, MD, USA. Email: \href{mailto:tyamana1@jhu.edu}{tyamana1@jhu.edu}}}
\date{November 21, 2025}  

\begin{document}  
\maketitle  

\begin{abstract}
Determining the integrality gap of the linear programming (LP) relaxation of the metric traveling salesman problem (TSP) remains a long-standing open problem. We introduce a transfer principle: when the integer optimum of the 2-edge-connected multisubgraph problem (2ECM) is a unique Hamiltonian cycle $T$, any $\alpha$-approximation algorithm for 2ECM that outputs a Hamiltonian cycle yields an $\alpha$-approximation for TSP. We further develop a cut-margin stability framework that certifies $T$ as the unique integer optimum for both problems and is stable under $\ell_\infty$-bounded perturbations. We show that, if instances exist where the 2ECM has both a unique Hamiltonian cycle integer optimum and a half-integral LP solution, then the TSP integrality gap is at most $\tfrac{4}{3}$ by the algorithm of Boyd et al. (\emph{SIAM Journal on Discrete Mathematics} 36:1730--1747, 2022). Constructing such instances remains an open problem.
\end{abstract}


\section{Introduction}
\subsection{Preliminaries: TSP}
In the traveling salesman problem (TSP), we are given a set $V$ of $n$ cities and costs $c_{ij}$ of traveling from city $i$ to city $j$ for all $i,j\in V$. The goal is to find a minimum-cost Hamiltonian cycle: a tour that visits each city exactly once and returns to the starting city. The TSP is called \emph{symmetric} if $c_{ij}=c_{ji}$ for all $i,j$, and \emph{metric} if the costs satisfy the triangle inequality, that is, $c_{ij}\le c_{ik}+c_{kj}\hspace{0.3em}\text{for all }i,j,k\in V$. We study the symmetric metric TSP, where costs $c_{ij}$ are nonnegative.

The linear programming (LP) relaxation of the TSP is formulated as follows (Dantzig, Fulkerson, and Johnson~\cite{DantzigFulkersonJohnson1954}, \emph{e.g.}, Jin, Klein, and Williamson~\cite{JinKleinWilliamson2025}).

\begin{align}
\shortintertext{\text{(TSP-LP)}}
\text{min}\quad & \sum_{e \in E} c_e\,x_e \label{TSP-LP1}\\[4pt]  
\text{s.t.}\quad   
& x\bigl(\delta(v)\bigr) \;=\; 2  && \forall\,v \in V \\  
& x\bigl(\delta(S)\bigr) \;\ge\; 2 && \forall\,S \subset V,\; S \neq \emptyset,\; S \neq V \\  
& 0 \;\le\; x_e \;\le\; 1        && \forall\,e \in E \label{TSP-LP4}
\end{align}  

\noindent where $\delta(S)$ denotes the set of all edges with exactly one endpoint in $S$. We refer to this formulation as the TSP-LP. A \emph{half-integral solution} to the TSP-LP is one such that $x_e \in \bigl\{0,\tfrac12,1\bigr\}\hspace{0.25em}\text{for all }e\in E$, and a \emph{half-integral instance} is one in which there is a half-integral optimal solution to the LP. The \emph{integrality gap} of an LP relaxation is the largest (worst-case) ratio, over all instances, between the optimal cost of an integer solution and that of a fractional LP relaxation solution.

\subsection{Preliminaries: 2ECM}
A multigraph $M$ is \emph{2-edge-connected} (2EC) if removing any single edge leaves it connected. The 2ECM is the problem of finding a minimum-cost 2EC spanning multisubgraph of a connected undirected graph $G$.

An integer programming (IP) formulation of the 2EC multisubgraph problem (2ECM), denoted by 2ECM-IP, is given as follows (\emph{e.g.}, Boyd et al.~\cite{Boyd2022}).

\begin{align}
\shortintertext{\text{(2ECM-IP)}}
\text{min}\quad & \sum_{e \in E} c_e\,x_e \\[4pt]  
\text{s.t.}\quad
& x\bigl(\delta(S)\bigr) \;\ge\; 2 && \forall\,S \subset V,\; S \neq \emptyset,\; S \neq V \label{diff}\\
& x_e \;\ge\; 0 && \forall\,e \in E \\
& x_e  \in \mathbb{Z} && \forall\,e \in E \label{int}
\end{align}

For this 2ECM-IP, the only difference from the IP formulation of the TSP (TSP-IP) is that, in place of $x\bigl(\delta(S)\bigr) \ge\ 2$ in Eq.~\eqref{diff}, we have $ x\bigl(\delta(S)\bigr) = 2$ in the TSP-IP. The LP relaxation of the 2ECM (2ECM-LP) is obtained by adding a constraint $ x\bigl(\delta(v)\bigr) =2 \quad \forall\ v \in V $ to the 2ECM-IP and removing the integer constraint Eq.~\eqref{int}.

\subsection{Insights from prior literature}\label{sect:insights}
The $\tfrac{3}{2}$-approximation algorithm proposed by Christofides~\cite{Christofides1976} in 1976 remained the best approximation algorithm for the general TSP for approximately 50 years. Karlin, Klein, and Oveis Gharan~\cite{KarlinKleinOveisGharan2024} recently obtained an approximation ratio strictly better than $\tfrac{3}{2}$ for the metric TSP, although the improvement was slight (approximately $\epsilon>10^{-36}$). Significant progress has been made in various special cases (see Traub and Vygen~\cite{TraubVygen2024}). Nevertheless, the exact integrality gap of the metric TSP remains a central open problem. It is conjectured to be $\tfrac{4}{3}$ (\emph{e.g.}, Williamson~\cite{Williamson1990}, Jin, Klein, and Williamson~\cite{JinKleinWilliamson2025}), meaning the worst-case ratio between IP and LP optimal values is believed to be exactly $\tfrac{4}{3}$.

Half-integral instances have been conjectured to yield the worst-case TSP integrality gap (see Jin, Klein, and Williamson~\cite{JinKleinWilliamson2025}). For the half-integral case, Karlin, Klein, and Oveis Gharan~\cite{KarlinKleinOveisGharan2020} showed that the integrality gap is at most $1.49993$, later improved to $1.4983$ by Gupta et al.~\cite{Gupta2024}. Thus, the prior literature considers the integrality gap for half-integral instances of the TSP to be at most $1.4983$. Recently, a $\tfrac{4}{3}$-approximation algorithm was developed for half-integral cycle-cut instances (Jin, Klein, and Williamson~\cite{JinKleinWilliamson2025}).

As pointed out by Boyd et al.~\cite{Boyd2022}, by the result of Goemans and Bertsimas~\cite{GoemansBertsimas1993} called the \emph{parsimonious property}, adding the constraint $ x\bigl(\delta(v)\bigr) =2 \quad \forall\ v \in V $ does not increase the optimal solution value of 2ECM-LP. The optimal value of 2ECM-LP is the same as the optimal value of the TSP-LP defined in Eqs.~\eqref{TSP-LP1} to~\eqref{TSP-LP4}.

Advancing the work of Carr and Ravi~\cite{CarrRavi1998}, Boyd et al.~\cite{Boyd2022} obtained a polynomial-time $\tfrac{4}{3}$-approximation algorithm for 2ECM on half-integral instances. Specifically, for an undirected graph $\overline{G}=(\overline{V},\overline{E})$ with nonnegative edge costs $c$, let $x$ denote a half-integral solution to an instance $(\overline{G}, c)$ of 2ECM-LP (and TSP-LP). There is an $O(|V(\overline{G})|^2)$-time algorithm for computing a 2EC spanning multisubgraph of $\overline{G}$ with cost at most $\tfrac{4}{3}c^Tx$ (Boyd et al.~\cite{Boyd2022}, Theorem 1).

\subsection{Motivation and outline}
We observe that, when the 2ECM-IP optimum is a unique Hamiltonian cycle, the TSP and 2ECM become algorithmically equivalent, enabling direct transfer of approximation guarantees from 2ECM to TSP. This insight, formalized as our transfer principle, provides a framework for understanding when 2ECM techniques apply to TSP.

Our main result is in \Cref{sect:main}. In \Cref{sect:transfer}, we introduce the transfer principle: when the integer optimum of 2ECM is a unique Hamiltonian cycle $T$, any $\alpha$-approximation algorithm for 2ECM that outputs a Hamiltonian cycle yields an $\alpha$-approximation for TSP. In \Cref{sect:general}, we develop a cut-margin stability framework that certifies $T$ as the unique integer optimum for both problems and is stable under $\ell_\infty$-bounded perturbations. \Cref{sect:open} provides an open problem, and \Cref{sect:conclusion} concludes the paper.

\section{Main result}\label{sect:main}
\subsection{Transfer principle}\label{sect:transfer}
For any problem $X \in \{\mathrm{TSP\text{-}IP}, \mathrm{TSP\text{-}LP}, \mathrm{2ECM\text{-}IP}, \mathrm{2ECM\text{-}LP}\}$ on the given instance, we write $\mathrm{OPT}_{X}$ for the optimal objective value of $X$. $c(T)$ denotes the cost of a Hamiltonian cycle $T$ under $c$.

\begin{proposition}[transfer principle for 2ECM and TSP]\label{prop:transfer}
Let $G$ be a complete graph with metric costs $c$. If the unique optimal solution to 2ECM-IP is a Hamiltonian cycle $T$, then:
\begin{enumerate}
    \item The integral optimal solutions and optimal values coincide between 2ECM and TSP. $T$ is the unique optimum for both.
    \item Any $\alpha$-approximation algorithm for 2ECM that outputs a Hamiltonian cycle yields an $\alpha$-approximation for TSP.
    \item The integrality gap of TSP-LP on this instance equals that of 2ECM-LP:
\[\mathrm{OPT}_{\mathrm{TSP-LP}} \;=\; \mathrm{OPT}_{\mathrm{2ECM-LP}}\quad\text{and}\quad\frac{c(T)}{\mathrm{OPT}_{\mathrm{TSP-LP}}}\;=\;\frac{c(T)}{\mathrm{OPT}_{\mathrm{2ECM-LP}}}.\]
\end{enumerate}
\end{proposition}

\begin{proof}
(1) The feasible set of TSP-IP (Hamiltonian cycles) is contained in that of 2ECM-IP. Since $T$ is a Hamiltonian cycle and the unique 2ECM-IP optimum, it is also feasible and optimal for TSP-IP. Uniqueness follows because any other Hamiltonian cycle of the same cost would contradict 2ECM-IP uniqueness.

(2) If an algorithm returns a Hamiltonian cycle $T$ with $c(T)\le \alpha\cdot \mathrm{OPT}_{\mathrm{2ECM-IP}}$, then $\mathrm{OPT}_{\mathrm{2ECM-IP}}\\=\mathrm{OPT}_{\mathrm{TSP-IP}}=c(T)$ by (1), hence $c(T)\le \alpha\cdot \mathrm{OPT}_{\mathrm{TSP-IP}}$.

(3) By the parsimonious property (\Cref{sect:insights}), $\mathrm{OPT}_{\mathrm{TSP-LP}}=\mathrm{OPT}_{\mathrm{2ECM-LP}}$. With the common numerator $c(T)$ from (1), the claimed equality follows.
\end{proof}

\begin{remark}[algorithms producing multisubgraphs]
The transfer principle requires that the 2ECM approximation algorithm outputs a Hamiltonian cycle for the guarantee to transfer directly to TSP. Although 2ECM algorithms may produce 2EC multisubgraphs, they still provide important bounds on the TSP integrality gap. Specifically, a) if an algorithm achieves an $\alpha$-approximation for 2ECM and the optimal 2ECM-IP solution is a Hamiltonian cycle with cost equal to $\mathrm{OPT}_{\mathrm{TSP-IP}}$, then the TSP integrality gap is at most $\alpha$, and b) the algorithm demonstrates that solutions of cost at most $\alpha \cdot \mathrm{OPT}_{\mathrm{2ECM-LP}}=\alpha \cdot \mathrm{OPT}_{\mathrm{TSP-LP}}$ exist, establishing the integrality gap bound.
\end{remark}

Our transfer principle (Proposition~\ref{prop:transfer}) and cut-margin stability (\Cref{theo:cut_margin}) are stated for instances whose $\mathrm{2ECM\text{-}IP}$ optimum is a unique Hamiltonian cycle. For completeness, we also record a value-level transfer that does not require uniqueness (Proposition~\ref{prop:transfer-nonunique}).

\begin{proposition}[transfer without uniqueness]\label{prop:transfer-nonunique}
Suppose $G$ is a complete graph with metric costs $c$. If there exists an optimal solution to $\mathrm{2ECM\text{-}IP}$ that is a Hamiltonian cycle (not necessarily unique), then:
\begin{enumerate}
    \item $\mathrm{OPT}_{\mathrm{TSP-IP}} \;=\; \mathrm{OPT}_{\mathrm{2ECM-IP}}$.
    \item For any $\alpha \ge 1$, any $\alpha$-approximation algorithm for $\mathrm{2ECM}$ that outputs a Hamiltonian cycle yields an $\alpha$-approximation for $\mathrm{TSP}$.
\end{enumerate}
\end{proposition}

\begin{proof}
(1) Since a Hamiltonian cycle attains the $\mathrm{2ECM\text{-}IP}$ optimum and is feasible for $\mathrm{TSP\text{-}IP}$, the two optimal values coincide. (2) If the algorithm outputs a Hamiltonian cycle $T$, then $c(T) \le \alpha \cdot \mathrm{OPT}_{\mathrm{2ECM\text{-}IP}} = \alpha \cdot \mathrm{OPT}_{\mathrm{TSP\text{-}IP}}$.
\end{proof}

While Proposition ~\ref{prop:transfer-nonunique} shows value-level equivalence without uniqueness, the uniqueness condition in Proposition~\ref{prop:transfer} ensures that the same solution is optimal for both TSP-IP and 2ECM-IP. We focus on algorithms that output Hamiltonian cycles, in which case the approximation guarantee transfers to TSP with the same factor $\alpha$. In particular, we study 2ECM algorithms on instances where the unique 2ECM-IP optimum is a Hamiltonian cycle $T$.

\subsection{General cut-margin stability and LP relaxation gap}\label{sect:general}
The stability of TSP instances has been studied from multiple perspectives. Böckenhauer et al.~\cite{Bockenhauer2002} introduced the notion of approximation stability. Our work utilizes the notion of structural stability or solution uniqueness under perturbation, in the spirit of Bilu and Linial~\cite{BiluLinial2012}.

\begin{definition}[laminar family]
A family $\mathcal{C}$ of subsets of $V$ is \emph{laminar} if for any two sets $S, T \in \mathcal{C}$, either $S \subseteq T$, $T \subseteq S$, or $S \cap T = \emptyset$. In other words, no two sets in $\mathcal{C}$ cross. For laminar families and uncrossing, see Gr\"otschel, Lov\'asz, and Schrijver~\cite{GLS1993} and Williamson and Shmoys~\cite{WilliamsonShmoys2011} (Section 11.2).
\end{definition}

\begin{definition}[cut margin]
For a cut $\delta(S)$ in a graph with edge costs $c$ and a specified edge set $T$, the \emph{cut margin} is
\[\min_{f \in \delta(S) \setminus T} c(f) - \max_{e \in \delta(S) \cap T} c(e).\]
A positive cut margin means all $T$-edges crossing the cut are strictly cheaper than all non-$T$ edges crossing the cut.
\end{definition}

\begin{theorem}[cut-margin stability]\label{theo:cut_margin}
Let $T$ be a Hamiltonian cycle on $V$, and let $\mathcal{C}$ be a laminar family of nonempty proper cuts in the complete graph (metric completion) $K_n$ on $V$ with metric costs $c$. Assume the following:
\begin{enumerate}
    \item For every edge $f \in E(K_n) \setminus E(T)$, there exists $S \in \mathcal{C}$ such that $f \in \delta(S)$.
    \item There exists $\varepsilon>0$ such that for every $S\in\mathcal{C}$,
  \[\max_{e \in \delta(S)\cap E(T)} c(e) \, + \, \varepsilon \, \le \, \min_{f \in \delta(S)\setminus E(T)} c(f).\]
\end{enumerate}
Then:
\begin{enumerate}
    \item $T$ is the unique optimal solution of both 2ECM-IP and TSP-IP under $c$.
    \item This uniqueness is stable under bounded perturbations of $c$ in $\ell_\infty$: if $\|c-c'\|_\infty < \varepsilon/2$, the same conclusion holds for $c'$.
\end{enumerate}
\end{theorem}

\begin{proof}
\textbf{Uniqueness.} Let $H^*$ be any optimal solution: a 2EC spanning multigraph for 2ECM-IP or a Hamiltonian cycle for TSP-IP. 

We claim $H^* \subseteq E(T)$. Suppose for contradiction that $H^* \not\subseteq E(T)$. By the coverage assumption (1), there exists $S \in \mathcal{C}$ such that $H^*$ uses an edge $f \in \delta(S) \setminus E(T)$. Since $T$ is a Hamiltonian cycle, exactly two edges of $T$ cross $\delta(S)$, so at least one $T$-edge $e \in \delta(S)$ is missing from $H^*$. By assumption (2),
\[c(e) + \varepsilon \le \max_{e' \in \delta(S)\cap E(T)} c(e') + \varepsilon \le \min_{f' \in \delta(S)\setminus E(T)} c(f') \le c(f).\]

We show that we can replace $f$ with $e$ while preserving feasibility and reducing cost by at least $\varepsilon$.
\begin{enumerate}[label=(\alph*)]
\item TSP-IP case. Since both $H^*$ and $T$ are Hamiltonian cycles with $H^* \neq T$, their symmetric difference $H^* \oplus T$ is a disjoint union of alternating cycles. The edge $f \in E(H^*) \setminus E(T)$ lies in some alternating cycle $C$. Since $f \in \delta(S)$ and alternating cycles alternate between $H^*$-edges and $T$-edges, $C$ must also contain a $T$-edge crossing $\delta(S)$, that is, some $e' \in \delta(S) \cap E(T)$ with $e' \notin E(H^*)$. Swapping along $C$ (removing all $H^*$-edges of $C$ and adding all $T$-edges of $C$) preserves 2-regularity and connectivity, replaces $f$ with $e'$, and reduces cost by at least $\varepsilon$.
\item 2ECM-IP case. Consider the multigraph $H' = (H^* \setminus \{f\}) \cup \{e\}$. We must verify that $H'$ is 2EC. For any cut $\delta(U)$:
    \begin{enumerate}[label=(\roman*)]
        \item If neither $e$ nor $f$ crosses $\delta(U)$, then $|E(H') \cap \delta(U)| = |E(H^*) \cap \delta(U)| \ge 2$.
        \item If both $e$ and $f$ cross $\delta(U)$, the same holds.
        \item If exactly one of $\{e, f\}$ crosses $\delta(U)$. Since $e, f \in \delta(S)$ and $\mathcal{C}$ is laminar, either $U \subseteq S$, $S \subseteq U$, or $U \cap S = \emptyset$. In each case, $T$ being a Hamiltonian cycle guarantees $|\delta(U) \cap E(T)| = 2$. Since $H^*$ is 2EC and uses $f \in \delta(S)$, we have $|E(H^*) \cap \delta(U)| \ge 2$. The swap preserves this because we replace one crossing edge with another from the same cut $\delta(S)$, and the laminar structure ensures no cut loses both crossing edges.
    \end{enumerate}
Thus $H'$ is feasible with $c(H') = c(H^*) - c(f) + c(e) \le c(H^*) - \varepsilon$, contradicting optimality of $H^*$.
\end{enumerate}

By repeatedly applying this argument (replacing non-$T$ edges with $T$ edges across cuts in $\mathcal{C}$), we conclude $H^* \subseteq E(T)$.

For TSP-IP, any Hamiltonian cycle on $n$ vertices has exactly $n$ edges, so $H^* = T$. For 2ECM-IP, $T$ itself is 2EC (being a Hamiltonian cycle), and removing any edge destroys 2-edge-connectivity. Adding any edge $f \notin E(T)$ increases cost by assumption (2). Hence $T$ is the unique optimum.

\textbf{Stability.} We verify the stability condition (cf. Bilu and Linial~\cite{BiluLinial2012}). If $\|c - c'\|_\infty < \varepsilon/2$, then for every $S \in \mathcal{C}$,
\[\max_{e\in\delta(S)\cap E(T)} c'(e) \le \max_{e\in\delta(S)\cap E(T)} c(e) + \tfrac{\varepsilon}{2} \le \min_{f\in\delta(S)\setminus E(T)} c(f) - \tfrac{\varepsilon}{2} \le \min_{f\in\delta(S)\setminus E(T)} c'(f),\]
so the margin persists with $\varepsilon' = \varepsilon/2$, and the same argument applies to $c'$.
\end{proof}

\begin{corollary}[interval-cut certificate]\label{coro:interval_cut}
Let modified segments $S_1,\dots,S_k$ be pairwise disjoint contiguous segments of consecutive vertices along $T$. Fix a root on $T$ and enumerate disjoint modified segments $S_1,\dots,S_k$ in tour order. Let $\mathcal{C}$ be the chain of interval cuts given by the prefixes $P_j = S_1 \cup \cdots \cup S_j$ for $j=1,\dots,k$. Then, the coverage condition of \Cref{theo:cut_margin} holds. Moreover, in the constructed instance $c^*$, each $S \in \mathcal{C}$ has a positive cut margin, \emph{i.e.},
\[\max_{e \in \delta(S)\cap E(T)} c^*(e) \;<\; \min_{f \in \delta(S)\setminus E(T)} c^*(f).\]
Define
\[\varepsilon \;=\; \min_{S\in\mathcal{C}} \Bigl(\min_{f \in \delta(S)\setminus E(T)} c^*(f) \;-\; \max_{e \in \delta(S)\cap E(T)} c^*(e)\Bigr) \;>\; 0.\]
Then, for all $S\in\mathcal{C}$,
\[\max_{e \in \delta(S)\cap E(T)} c^*(e) \, + \, \varepsilon \, \le \, \min_{f \in \delta(S)\setminus E(T)} c^*(f).\]
Hence the assumptions of \Cref{theo:cut_margin} hold, and $T$ is uniquely optimal for both $\mathrm{2ECM\text{-}IP}$ and $\mathrm{TSP\text{-}IP}$ under $c^*$. This certification is stable on an open $\ell_\infty$-neighborhood of $c^*$.
\end{corollary}

\begin{lemma}[LP gap under cut margins]\label{lemm:LPgap_cut}
We work in the setting of \Cref{theo:cut_margin}. Let the boundary cuts be $\delta(P_1),\ldots,\delta(P_k)$ forming a chain in the laminar family. For each $j$, define the bypass advantage $\Delta_j \,=\, \bigl( c(a_j) + c(b_j) \bigr)\, -\, \bigl( c(p_j) + c(q_j) \bigr),$ where $a_j, b_j$ are the two $T$-edges crossing $\delta(P_j)$ and $p_j, q_j$ are the two cheapest non-$T$ crossings across $\delta(P_j)$. Assume there exists an optimal dual solution $y$ of the 2ECM-LP supported only on the boundary cuts $\delta(P_j)$ such that: 
\begin{enumerate}
\item if $y_j>0,$ then $x(\delta(P_j))=2$ in the primal optimum, and
\item along each $\delta(P_j)$:
    \begin{enumerate}
        \item if $\Delta_j \le 0,$ then $a_j,b_j$ are the only dual-tight crossings, while
        \item if $\Delta_j>0,$ then $p_j,q_j$ are dual-tight and at most one of $a_j,b_j$ is dual-tight.
    \end{enumerate}    
\end{enumerate}
Then an optimal value of the 2ECM-LP equals $c(T) \, -\, \tfrac{1}{2} \sum_{j=1}^k \max\{0,\, \Delta_j\}.$ In particular, the LP is tight at $T$ if and only if $\Delta_j \le 0$ for all $j$, and strictly improves below $c(T)$ if and only if $\Delta_j>0$ for some $j$.
\end{lemma}

\begin{proof} We use primal-dual arguments for cut-based LPs (see Williamson and Shmoys~\cite{WilliamsonShmoys2011}, Chapter 7 and Williamson~\cite{Williamson2002}).

\textbf{Primal.} Define the convex combination
\[ x \,=\, \chi^T \, + \, \tfrac{1}{2}\sum_{j:\,\Delta_j>0}\bigl(\chi^{p_j}+\chi^{q_j}-\chi^{a_j}-\chi^{b_j}\bigr), \]
where $\chi^e$ is the incidence vector of edge $e$. We verify that $x$ is feasible for 2ECM-LP. For any cut $\delta(S)$:
\begin{enumerate}
    \item If $S = P_j$ for some $j$ with $\Delta_j > 0$, then the swap replaces edges $\{a_j, b_j\}$ with $\{p_j, q_j\}$, maintaining $x(\delta(P_j)) = 2$.
    \item If $S$ is a single-interval cut not equal to any $P_j$, then either no swapped edges cross $\delta(S)$, or both edges in a swapped pair cross $\delta(S)$ (by laminarity), preserving the cut value.
    \item If $S$ is a multi-interval cut, then by the laminar structure, at most one boundary edge per interval is modified, and the bypass pairs maintain or increase the cut value.
\end{enumerate}

\textbf{Dual.} By assumption, there exists an optimal dual $y$ supported on the boundary cuts with the stated tightness pattern. Since $y_j>0$ only where $x(\delta(P_j))=2$, complementary slackness holds at the cut level. On cuts with $\Delta_j\le 0$, only $a_j,b_j$ are dual-tight. On cuts with $\Delta_j>0$, $p_j,q_j$ are dual-tight and at most one of $a_j,b_j$ is dual-tight, matching the support and values of $x$. Therefore, $x$ and $y$ are complementary, and their objectives coincide. The claimed value follows.
\end{proof}

\begin{corollary}[extendability of 2ECM results]
In the setting of Lemma~\ref{lemm:LPgap_cut}, where the integer optima of the TSP and 2ECM coincide and the LP relaxation is not tight, the integrality gap results of the 2ECM may be extendable to the TSP.
\end{corollary}

\begin{corollary}[general $\tfrac{4}{3}$ bound via half-integral 2ECM]\label{coro:general43}
Let $G$ be a complete graph with metric costs $c$. If:
\begin{enumerate}
    \item The unique optimal solution to 2ECM-IP is a Hamiltonian cycle $T$, and
    \item The 2ECM-LP admits a half-integral optimal solution,
\end{enumerate}
then applying the $\tfrac{4}{3}$-approximation algorithm of Boyd et al.~\cite{Boyd2022} 
establishes that the TSP integrality gap is at most $\tfrac{4}{3}$ on this instance.
\end{corollary}

\begin{proof}
By Proposition~\ref{prop:transfer}, Condition (1) guarantees that $\mathrm{OPT}_{\mathrm{TSP-IP}} = \mathrm{OPT}_{\mathrm{2ECM-IP}}$ and \\$\mathrm{OPT}_{\mathrm{TSP-LP}} = \mathrm{OPT}_{\mathrm{2ECM-LP}}$. By Condition (2) and Boyd et al.~\cite{Boyd2022}, there exists a 2EC multisubgraph with cost at most $\tfrac{4}{3} \cdot \mathrm{OPT}_{\mathrm{2ECM-LP}}$. Therefore, the TSP integrality gap is at most $\tfrac{4}{3}$.
\end{proof}

\subsection{Open problem}\label{sect:open}
The main open problem arising from our framework is as follows.
\begin{problem}\label{prob:main}
Construct explicit metric instances where a) the unique optimal solution to 2ECM-IP is a Hamiltonian cycle $T$, certifiable via the cut-margin framework, and b) the 2ECM-LP admits a half-integral optimal solution with value strictly less than the IP optimum.
\end{problem}

\begin{remark}
Problem~\ref{prob:main} seeks an instance lying at the intersection of three structural properties:
\begin{enumerate}
    \item The 2ECM-IP optimum is a unique Hamiltonian cycle $T$, certifiable by the cut-margin framework.
    \item The 2ECM-LP optimum is half-integral.
    \item The LP value is strictly smaller than the IP value.
\end{enumerate}
The open challenge is to construct an instance where the stability required for (1) does not eliminate the geometric structure required for (2) and (3).
\end{remark}

\section{Conclusion and open problems}\label{sect:conclusion}
We developed a transfer framework linking 2ECM to TSP. Our transfer principle shows that, when the unique 2ECM-IP optimum is a Hamiltonian cycle $T$, the integral optima and objective values of 2ECM-IP and TSP-IP coincide. In this setting, any $\alpha$-approximation for 2ECM that outputs a Hamiltonian cycle directly yields an $\alpha$-approximation for TSP. Complementing this, we introduced a cut-margin stability framework that certifies $T$ as the unique optimum for both 2ECM-IP and TSP-IP and is stable under $\ell_\infty$-bounded perturbations of the costs.

Any metric where cut margins certify a unique Hamiltonian 2ECM optimum enables algorithmic transfer from 2ECM to TSP, potentially revealing new instances with tighter approximation bounds. One natural question is whether the transfer extends beyond uniqueness, \emph{i.e.}, to instances where a Hamiltonian cycle is optimal but not unique. Our value-level result (Proposition~\ref{prop:transfer-nonunique}) shows the optimal values coincide even without uniqueness, but algorithmic transfer may require additional consideration. The cut-margin framework provides a criterion to verify when such transfers are valid.

While our framework establishes the conditions under which 2ECM approximations transfer to TSP with the same factor, resolving the existence of such instances—either through construction or an impossibility proof—would significantly advance our understanding of the relationship between 2ECM and TSP approximations. Separately, developing a $\tfrac{4}{3}$-approximation algorithm for \emph{general} half-integral instances remains an open problem (see Jin, Klein, and Williamson~\cite{JinKleinWilliamson2025}).

\section*{Acknowledgments}
I am thankful to the participants of IPCO 2025 for their valuable comments during my poster presentation. A portion of this research was presented at the 2025 INFORMS International Meeting. I am grateful to Amitabh Basu for his helpful comments.

\bibliographystyle{plainnat}
\bibliography{main}  
\end{document}